\documentclass[a4paper,11pt]{amsart}
\usepackage{mathrsfs}
\usepackage{appendix}
\usepackage{amssymb}
\usepackage{fancyhdr}
\usepackage{charter} 
\usepackage{typearea}
\usepackage{color}
\usepackage{amsmath,amstext,amsthm,amscd}
\usepackage{hyperref}

\usepackage[a4paper,top=2.5cm,bottom=2.5cm,left=2cm,right=2cm]{geometry}

\makeatletter
      \def\@setcopyright{}
      \def\serieslogo@{}
      \makeatother

\newcommand{\N}{\mathbb N}
\newcommand{\dbar}{\partial}
\newcommand{\ddbar}{\overline\partial}

\newcommand{\ol}{\overline}

\DeclareMathOperator{\Ker}{Ker}

\newcommand{\cali}[1]{\mathscr{#1}}
\newcommand{\cH}{\cali{H}}
\newcommand{\cE}{\cali{E}}

\newcommand{\R}{\mathbb{R}}

\newcommand{\ov}{\overline}

\def\Im{{\rm Im}}

\DeclareMathOperator{\supp}{supp}
\DeclareMathOperator{\Dom}{Dom}
\DeclareMathOperator{\rank}{rank}
 \def\cC{\mathscr{C}}
\def\cL{\mathscr{L}}
\theoremstyle{plain}
\newtheorem{theorem}{Theorem}[section]
\newtheorem{lemma}[theorem]{Lemma}
\newtheorem{corollary}[theorem]{Corollary}
\newtheorem{proposition}[theorem]{Proposition}

\newtheorem{definition}[theorem]{Definition}

\numberwithin{equation}{section}

\begin{document}

\title[Strong holomorphic Morse inequalities on non-compact complex manifolds with optimal fundamental estimate]
{Strong holomorphic Morse inequalities on non-compact complex manifolds with optimal fundamental estimate}
 
\author[]{Manli Liu}
\address{Department of Mathematics, South China Agricultural University, Guangzhou 510642, China}
\email{Liuml@scau.edu.cn}

\author[]{Guokuan Shao}
\address{School of Mathematics (Zhuhai), Sun Yat-sen University, Zhuhai 519082, Guangdong, China}
\thanks{Guokuan Shao is supported by NSFC Grant No.12471082 and Natural Science Foundation of Guangdong Province Grant No.2024A1515011223
.}
\email{shaogk@mail.sysu.edu.cn}

\author[]{Wenxuan Wang}
\address{School of Mathematics (Zhuhai), Sun Yat-sen University, Zhuhai 519082, Guangdong, China}
\email{wangwx67@mail2.sysu.edu.cn}

\date{}

\begin{abstract}
In this paper, we establish strong holomorphic Morse inequalities on non-compact manifolds under the condition of optimal fundamental estimates. We show that optimal fundamental estimates are satisfied and then strong holomorphic Morse inequalities hold true in various settings.
\end{abstract}
\maketitle

\section{Introduction}
  
Holomorphic Morse inequalities was established by Demailly \cite{Dem:85}. One motivation was Siu's solution of Grauert-Riemenschneider conjecture \cite{Siu:84} and the proof was inspired by Witten's analytic proof of the classical Morse inequalities. Holomorphic Morse inequalities are global results which encode local datas, which can be studied by the behaviors of heat kernels, Bergman kernels or Szegő kernels. It provides a flexible way to produce holomorphic sections of high tensor powers of line bundle under mild positivity assumptions. Variants of the holomorphic Morse inequalities have been intensively studied during recent years, see \cite{MM07} and the references therein for a comprehensive exposition.  The inequalities have profound applications in complex geometry and algebraic geometry. Morse inequalities have been generalized to new settings, such as \cite{HM12,HsiaoLi:16,HHLS20} for CR manifolds by using CR scaling technique and complex manifolds with boundary by developing reduction to boundary technique.

Recently Li-Shao-Wang \cite{LSW} introduced a new concept of optimal fundamental estimate
and gave a unified proof of the weak holomorphic Morse inequalities on various settings of non-compact manifolds. Moreover, they established asymototics of spectral function of lower energy forms and proved versions of weak
holomorphic Morse inequalities for lower energy forms on complete Hermitian manifolds. Based on the previous work, Peng-Shao-Wang \cite{PSW} studied weak holomorphic Morse inequalities for lower energy forms 
on weakly $1$-complete manifolds and $q$-convex manifolds. 

Note that the main results in \cite{LSW, PSW} focused only on weak holomorphic Morse inequalities. The motivation of this paper is to explore strong holomorphic Morse inequalities with optimal fundamental estimate. Let $(X,\omega)$ be a Hermitian manifold of dimension $n$ and $(L,h^L)$ a holomorphic Hermitian line bundle on $X$. Denote by $R^L$ the curvature form of $(L,h^L)$ and $c_1(L,h^L):=\frac{\sqrt{-1}}{2\pi}R^L$ the first Chern class of $(L,h^L)$. Our main theorem is the following.
\begin{theorem}[Strong holomorphic Morse inequalities]\label{prop_main}
		  Let $0\leq q\leq n$.
		Suppose there exist a compact subset $K\subset X$ and $C>0$ such that,
		for sufficiently large $k$, any $s\in \Dom(\ddbar_k)\cap \Dom(\ddbar^{*}_{k})\cap L^2_{0,j}(X,L^k)$, where $0\le j\le q$,
		\begin{equation}\label{eq_ofe}
		\left(1-\frac{C}{k}\right)||s||^2\leq \frac{C}{k}\left(||\ddbar_ks||^2+||\ddbar^{*}_{k}s||^2\right)+\int_{K} |s|^2 dV_X.
		\end{equation}
	Then  for any $0\le r\le q$, we have the estimate for alternative sum of the dimension of $L^2$-Dolbeault cohomology,
	\begin{equation*}
	\sum_{j=0}^r(-1)^{r-j}\dim_\mathbb{C}H_{(2)}^j(X,L^k)\leq \frac{k^n}{n!}\int_{K(\leq r)}{(-1)^r c_1(L,h^L)^n}+o(k^n),
	\end{equation*}
 where $K(\leq r):=\cup_{j=0}^r K(j)$ and $K(j):=\{x\in K: R^{L}_x 
  \ \text{has exactly}\ j \ \text{negative eigenvalues and}\ n-j\  \text{positive eigenvalues}\}$ . 
 If (\ref{eq_ofe}) holds for any $q\le j\le n$, then for any $q\le r\le n$, we have,
        \begin{equation*}
	\sum_{j=r}^n(-1)^{j-r}\dim_\mathbb{C}H_{(2)}^j(X,L^k)\leq \frac{k^n}{n!}\int_{K(\ge r)}{(-1)^r c_1(L,h^L)^n}+o(k^n),
	\end{equation*}
  where $K(\geq r):=\cup_{j=r}^n K(j)$. In particular, if (\ref{eq_ofe}) holds for any  $0\le j\le n$, we have the Riemann-Roch-Hirzebruch formula
   \begin{equation*}
	\sum_{j=0}^n(-1)^{j}\dim_\mathbb{C}H_{(2)}^j(X,L^k)= k^n\int_{K}\frac{c_1(L,h^L)^n}{n!}+o(k^n).
	\end{equation*}
 
\end{theorem}

In particular, we can deduce weak Morse inequalities from Theorem \ref{prop_main} as a corollary.
\begin{corollary}[Weak Morse inequalities {\cite[Theorem 1.1]{LSW}}]\label{weakMorse}
	Let $0\leq q\leq n$.
		Suppose there exist a compact subset $K\subset X$ and $C>0$ such that,
		for sufficiently large $k$, we have
		\begin{equation*}
		\left(1-\frac{C}{k}\right)||s||^2\leq \frac{C}{k}\left(||\ddbar_ks||^2+||\ddbar^{*}_{k}s||^2\right)+\int_{K} |s|^2 dV_X
		\end{equation*}
		for $s\in \Dom(\ddbar_k)\cap \Dom(\ddbar^{*}_{k})\cap L^2_{0,q}(X,L^k)$.
	Then we have
\begin{equation*}
\dim H^{q}_{(2)}(X,L^{k}) \leq\frac{k^{n}}{n!} \int_{K(q)}(-1)^q c_1(L,h^L)^n+o(k^{n}).
\end{equation*}
\end{corollary}

 With the optimal fundamental estimate \eqref{eq_ofe} at hands, combined with also asymptotic estimates of Bergman kernel functions, we can integrate the Bergman kernel function over a compact subset to get the dimension of harmonic spaces. The ``optimal" means the coefficient of the term $\int_K |s|^2 dv_X$ is $1$. Suppose the coefficient is less than $1$, the harmonic forms vanish as $k\rightarrow \infty$. So the coefficient $1$ represents the precise interface of the vanishing everywhere and the concentration on a compact subset for harmonic forms. See Sec. \ref{subsec_ofe} for definitions of the optimal and the usual fundamental estimate. It is remarkable that our approach is direct. Moreover, the upper bound of asymptotic dimension of cohomology are sharper than the previous in literatures (Explicitly we can replace relatively compact domain $U$ in \cite[(3.2.58)]{MM07} including $K$ inside  by $K$ itself). There are some results of optimal fundamental estimate for non-compact manifolds, such as  \cite{LSW, PSW}.  Based on these results, we establish strong holomorhpic Morse inequalities in various situations as follows and give a uniform and simple proof.

\begin{theorem} \label{thm_w1c}
	Let $X$ be a weakly $1$-complete manifold of dimension $n$. Let $(L,h^L)$ be a holomorphic Hermitian line bundle on $X$. Assume $K\subset X_c:=\{\varphi<c\}$ is a compact subset and $(L,h^L)$ is Griffith $q$-positive on $X\setminus K$ with $q\geq 1$. Then, there exits a Hermitian metric $\omega$ on $X$ and as $k\rightarrow \infty$, such that for any $q\le r\le n$,
	\begin{equation*}
		\sum_{j=r}^n(-1)^{j-r}\dim_\mathbb{C}H_{(2)}^j(X_c,L^k)\leq \frac{k^n}{n!}\int_{K(\ge r)}{(-1)^r c_1(L,h^L)^n}+o(k^n).
	\end{equation*}
	In particular, if $L>0$ on $X\setminus K$, for any $1\leq r\leq n$ we have, 
        \begin{equation*}
		\sum_{j=r}^n(-1)^{j-r}\dim_\mathbb{C}H^j(X,L^k)\leq \frac{k^n}{n!}\int_{K(\ge r)}{(-1)^r c_1(L,h^L)^n}+o(k^n).
	\end{equation*}
\end{theorem}
Note that Marinescu\cite{M:92} gives a proof of weak Morse inequalities (Corollary \ref{weakMorse}) in this setting, but the results can be improved to strong Morse inequalities using our methods.

\begin{theorem}[Ma-Marinescu\cite{MM07}]\label{thm_psc}
	Let $M\Subset X$ be a smooth pseudoconvex domain in a complex manifold $X$ of dimension $n$. Let $(L,h^L)$ be a holomorphic Hermitian line bundle on $X$.  Let $(L,h^L)$ be positive in a neighbourhood of the boundary $bM$ of $M$.
	Then there are a compact subset $K\Subset M$ and a Hermitian metric $\omega$ on $X$, such that for any $1\le r\le n$,
 \begin{equation*}
     \sum_{j=r}^n(-1)^{j-r}\dim_\mathbb{C}H_{(2)}^j(M,L^k)\leq \frac{k^n}{n!}\int_{K(\ge r)}{(-1)^r c_1(L,h^L)^n}+o(k^n).
 \end{equation*}
\end{theorem}

  \begin{theorem}[Bouche{\cite{Bou:89}}]\label{thm_q_convex}
  	Let $X$ be a $q$-convex manifold of dimension $n$ and $1\leq q \leq n$. Let $(L,h^L)$ be a holomorphic Hermitian line bundle on $X$. Suppose $R^L$ has at least $n-s+1$ non-negative eigenvalues on $X\setminus M$ for a compact subset $M$ with $1\leq s\leq n$.
  	Then for any $s+q-1\le r\le n$, we have 
  \begin{equation*}
     \sum_{j=r}^n(-1)^{j-r}\dim_\mathbb{C}H^j(X,L^k)\leq \frac{k^n}{n!}\int_{M(\ge r)}{(-1)^r c_1(L,h^L)^n}+o(k^n).
 \end{equation*}
  \end{theorem}

\begin{theorem}\label{thm_complete}
	Let $(X,\Theta)$ be a complete Hermitian manifold of dimension $n$. Let $(L,h^L)$ be a holomorphic Hermitian line bundle  on $X$ such that $\Theta=c_1(L,h^L)$ on $X\setminus M$ for a compact subset $M$. Then for any $1\le r\le n$, we have
  \begin{equation*}
     \sum_{j=r}^n(-1)^{j-r}\dim_\mathbb{C}H_{(2)}^j(X,L^k\otimes K_X)\leq \frac{k^n}{n!}\int_{M(\ge r)}{(-1)^r c_1(L,h^L)^n}+o(k^n).
 \end{equation*}
\end{theorem}Ma-Marinescu {\cite{MM07}} give a lower bound of $\dim_\mathbb{C}H_{(2)}^0(X,L^k\otimes K_X)$ in this case, but our methods could get strong Morse inequalities.

 This paper is organized as follows. In Sec. \ref{Sec_pre} we introduce some definitions and elementary facts related to $L^2$-cohomology and $L^2$ Hodge decomposition, some useful estimates will be mentioned.  In Sec. \ref{sec_Asy}, we recall Berman's scaling technique and prove a local version of strong Morse inequalities for non-compact manifolds. Strong holomorphic Morse inequalities (Theorem \ref{prop_main}) will be proved in Sec. \ref{sec_HMI}. In Sec. \ref{Sec_l2wmi} we give some applications and examples of our results and prove Theorem \ref{thm_w1c}--\ref{thm_complete}.

\section{Preliminaries and Notations} \label{Sec_pre}

Let $(X, \omega)$ be a Hermitian manifold of dimension $n$ and $(F, h^F)$ and $(L,h^L)$ be holomorphic Hermitian vector bundles on $X$ with $\rank(L)=1$.
Let $\Omega^{p,q}(X, F)$ be the space of smooth $(p,q)$-forms on $X$ with values in $F$ for $p,q\in \N$. If $\rank(F)=1$, the curvature of $(F, h^F)$ is defined by $R^F=\ddbar\dbar \log|s|^2_{h^{F}}$ for any local holomorphic frame $s$ and the Chern-Weil form of the first Chern class of $F$ is denoted by $c_1(F, h^F)=\frac{\sqrt{-1}}{2\pi}R^F$, which is a real $(1,1)$-form on $X$. The volume form is given by $dV_{X}:=\omega_n:=\frac{\omega^n}{n!}$.

We identify the two-form $R^F$ at every $x\in M$ with a linear endomorphism $Q_F$ given by
\begin{equation*}
    R^F_x (\alpha, \ol{\beta}) = \langle Q_F \alpha, \ol{\beta}\rangle_\omega,\; \forall \, \alpha, \beta\in T^{1,0}_xX. 
\end{equation*}

\subsection{$L^2$-coholomogy}
Let $\Omega^{p,q}_{0}(X, F)$ be the subspace of $\Omega^{p,q}(X, F)$ consisting of elements with compact support.
The $L^2$-scalar product on $\Omega^{p,q}_0(X, F)$ is given by
\begin{equation*}
(s_1,s_2)_X:=\int_X \langle s_1(x), s_2(x) \rangle_h dV_X(x)
\end{equation*}
where $\langle\cdot,\cdot\rangle_h:=
\langle\cdot,\cdot\rangle_{h^F,\omega}$
is the pointwise Hermitian inner product induced by $\omega$ and $h^F$.
We denote by $L^2_{p,q}(X, F)$, the $L^2$ completion of $\Omega^{p,q}_0(X, F)$.

Let $\ddbar^{F}: \Omega_0^{p,q} (X, F)\rightarrow L^2_{p,q+1}(X,F) $ be the Dolbeault operator and let  $ \ddbar^{F}_{\max} $ be its maximal extension. From now on we still denote the maximal extension by $ \ddbar^{F} :=\ddbar^{F}_{\max} $ and  the corresponding Hilbert space adjoint by $\ddbar^{F*}:=\ddbar^{F*}_H:=(\ddbar^{F}_{\max})_H^*$. We write $\ddbar^F_k:=\ddbar^{L^k\otimes F}$ for simplification.
Consider the complex of closed, densely defined operators
$L^2_{p,q-1}(X,F)\xrightarrow{T=\ddbar^{F}}L^2_{p,q}(X,F)\xrightarrow{S=\ddbar^{F}} L^2_{p,q+1}(X,F)$,
then $(\ddbar^{F})^2=0$. By \cite[Proposition 3.1.2]{MM07}, the operator defined by
\begin{eqnarray*}\nonumber
\Dom(\square^{F})&=&\{s\in \Dom(S)\cap \Dom(T^*):
Ss\in \Dom(S^*),~T^{*}s\in \Dom(T) \}, \\
\square^{F}s&=&S^{*}Ss+TT^{*}s \quad \mbox{for}~s\in \Dom(\square^{F}),
\end{eqnarray*}
is a positive, self-adjoint extension of Kodaira Laplacian, called the Gaffney extension.
	The space of harmonic forms $\cH^{p,q}(X,F)$ is defined by
	\begin{equation*}
	\cH^{p,q}(X,F):=\Ker(\square^{F})\cap  L^2_{p,q}(X,F) =\{s\in \Dom(\square^{F})\cap L^2_{p,q}(X, F): \square^{F}s=0 \}.
	\end{equation*}
	The $q$-th reduced (resp. non-reduced) $L^2$-Dolbeault cohomology are defined by, respectively,
	\begin{equation}\label{eq46}
	\overline{H}^{0,q}_{(2)}(X,F):=\dfrac{\Ker(\ddbar^{F})\cap  L^2_{0,q}(X,F) }{[ \Im( \ddbar^{F}) \cap L^2_{0,q}(X,F)]}, \quad H^{0,q}_{(2)}(X,F):=\dfrac{\Ker(\ddbar^{F})\cap  L^2_{0,q}(X,F) }{ \Im( \ddbar^{F}) \cap L^2_{0,q}(X,F)},
	\end{equation}
	where $[V]$ denotes the closure of the space $V$.
According to the general regularity theorem of elliptic operators,
$s\in \cH^{p,q}(X,F) $ implies $s\in\Omega^{p,q}(X,F)$.

\subsection{$L^2$-Weak Hodge decomposition}
It is obvious that $\cH^{0,q}(X,F)=\Ker(S)\cap\Ker(T^*)$ by the definition of $\square^F$. It is well known that
\begin{equation*}
\begin{split}
\Im(T)^{\perp}=&\Ker(T^*)
=(\Ker(T^*)\cap\Ker(S))\oplus (\Ker(T^*)\cap\Ker(S)^{\perp})\\
=&\cH^{0,q}(X,F)\oplus (\Ker(T^*)\cap\Ker(S)^{\perp})   
=\cH^{0,q}(X,F)\oplus (\Ker(T^*)\cap[\Im(S^*)])\\
=&\cH^{0,q}(X,F)\oplus [\Im(S^*)],
\end{split}
\end{equation*}
where $\Im(T)^{\perp}$ denotes  orthogonal complement of   $\Im(T)$ in $L^2_{0,q}(X,F)$ and the last equality comes from $T^{*}S^{*}=0$. 

We also have 
\begin{equation*}
\begin{split}
\Ker(S)=&([\Im(T)]\cap\Ker(S))\oplus (\Im(T)^{\perp}\cap\Ker(S))
=[\Im(T)]\oplus (\Im(T)^{\perp}\cap\Ker(S))\\
=&[\Im(T)]\oplus (\Ker(T^*)\cap\Ker(S))
=\cH^{0,q}(X,F)\oplus [\Im(T)],
\end{split}
\end{equation*}
the second equality comes from $ST=0.$ Since $L^2_{0,q}(X,F)= \Im(T)^{\perp}\oplus[\Im(T)]$, we get weak Hodge decomposition\cite[(3.1.21)]{MM07}
\begin{equation}\label{weak H}
    \begin{split}
        L^2_{0,q}(X,F)=&\cH^{0,q}(X,F)\oplus [\Im(\ddbar^{F*})]\oplus[\Im(\ddbar^F)],\\
        \Ker(\ddbar^F)=&\cH^{0,q}(X,F)\oplus [\Im(\ddbar^F)].
    \end{split}
\end{equation}
From weak Hodge decomposition (\ref{weak H}) and definition of $L^2$-Dolbeault cohomology group (\ref{eq46}), we get a canonical isomorphism
\begin{equation*}
\overline{H}^{0,q}_{(2)}(X,F)\cong \cH^{0,q}(X,F),
\end{equation*} which associates to each cohomology class its unique harmonic representative.  Set $\overline{H}^{q}_{(2)}(X,F):=\overline{H}^{0,q}_{(2)}(X,F)$ and $H^{q}_{(2)}(X,F):=H^{0,q}_{(2)}(X,F)$. The sheaf cohomology of holomorphic sections of $F$ is isomorphic to the Dolbeault cohomology, $H^\bullet(X,F)\cong H^{0,\bullet}(X,F)$. 

\subsection{Fundamental estimates and strong Hodge decomposition}\label{subsec_ofe}
We say the \textbf{fundamental estimate} holds in bidegree $(0,q)$ for forms with values in $F$ with $0\leq q\leq n$, if there exist a compact subset $K\subset X$ and $C>0$ such that, for $s\in \Dom(\ddbar^F)\cap\Dom(\ddbar^{F,*})\cap L^2_{0,q}(X,F)$, we have
\begin{equation*}
\|s\|^2\leq C\left(\|\ddbar^F s\|^2+\|\ddbar^{F*}s\|^2+\int_K|s|^2dV_X\right).
\end{equation*}

Then we claim the strong Hodge decomposition under the fundamental estimate.

\begin{theorem}[{\cite[Theorem 3.1.8]{MM07}}]
    If the fundamental estimate holds in bidegree $(0,q)$, then $\ddbar^F$ on $L_{0,q-1}^2(X,F)$ and $\square^F$ on $L_{0,q}^2(X,F)$ have closed range and we have the strong Hodge decomposition
    \begin{equation}\label{strong H}
    \begin{split}
        L^2_{0,q}(X,F)=&\cH^{0,q}(X,F)\oplus \Im(\ddbar^{F*})\oplus\Im(\ddbar^F),\\
        \Ker(\ddbar^F)=&\cH^{0,q}(X,F)\oplus \Im(\ddbar^F).
    \end{split}
\end{equation}
    Moreover, $\cH^{0,q}(X,F)$ is finite-dimensional. We have a canonical isomorphism
\begin{equation*}
    \mathscr{H}^{0,q}(X,F)\rightarrow H_{(2)}^{0,q}(X,F),\quad s\mapsto[s].
\end{equation*}
\end{theorem}

Define $H^{q}_{(2)}(X,F):=H^{0,q}_{(2)}(X,F)$, hence, $H^{q}_{(2)}(X,F)\cong\cH^{0,q}(X,F)$, under the fundamental estimate condition.
 We say the \textbf{optimal fundamental estimate} holds in bidegree $(0,q)$ for forms with values in $L^k$ with $0\leq q\leq n$, if there exist a compact subset $K\subset X$ and $C_0>0$ such that, for sufficiently large $k$ we have
for $s\in \Dom(\ddbar^F_k)\cap \Dom(\ddbar^{F*}_{k})\cap L^2_{0,q}(X,L^k)$,
\begin{equation*}
\left(1-\frac{C_0}{k}\right)||s||^2\leq \frac{C_0}{k}\left(||\ddbar^F_ks||^2+||\ddbar^{F*}_{k,H}s||^2\right)+\int_{K} |s|^2 dV_X.
\end{equation*}
Note that the condition of Theorem \ref{prop_main} is optimal fundamental estimate holds for some $(0,j)$ form. With respect to forms with values in $L^k$, optimal fundamental estimate holds implies fundamental estimate holds for sufficiently large $k$. The optimal means the coefficient of the term $\int_K |s|^2 dv_X$ is $1$, the above optimal fundamental estimate was introduced in {\cite{LSW}}.

\section{Asymptotics of Bergman Kernel Functions for Lower Energy Forms}\label{sec_Asy}
Berman proved a local version of weak holomorphic Morse inequalities, which holds regardless of compactness or completeness. Refer to \cite[Theorem 1.1, Remark 1.3]{Be04} and \cite[Corollary 1.4]{HM:14} for details. In this section, we use Berman's technique and prove a local version of strong Morse inequalities for non-compact manifolds. 
\subsection{Kernel function and extremal function}

Let $(X,\omega)$ be a Hermitian manifold of dimension $n$. Let $(L,h^L)$ be a holomorphic Hermitian line bundle on $X$.
Let $\{{s}^k_j\}_{j\geq 1}$ be an orthonormal basis of $\cH^{0,q}({X},{L}^k)$, $0\leq q\leq n$, and $|\cdot|:=|\cdot|_{{h}_k,{\omega}}$ the point-wise Hermitian norm. The \textbf{Bergman kernel function} on $X$ is defined by
\begin{equation*}
{B}^{q}_k({x}):=\sum_{j}|s^k_j({x})|_{h}, \;{x}\in {X}.
\end{equation*}
The \textbf{extremal function} on $X$ is defined by 
\begin{equation*}
{S}^{q}_k({x}):=\sup_{\alpha\in \cH^{q}(X,L^{k})}\frac{|\alpha(x)|^{2}_h}{||\alpha||^{2}},
\;{x}\in {X}.
\end{equation*}

Let $\square^E_k$ be the Gaffney extension of Kodaira Laplacian. Let $0\leq q\leq n$ and $\lambda\geq0$.
Let $E^q_{\leq \lambda}(\square_k): L^2_{0,q}(X,L^k)\rightarrow \cE^q(\lambda,\square_k):=\Im E^q_{\leq \lambda}(\square_k)$ be the spectral projection of $\square_k$. By the elliptic property of $\square_k$, it's not difficult to deduce $\cE^q(\lambda,\square_k)\subset \Omega^{0,q}(X,L^{k})$. Let $\{s_j\}_{j\geq 1}$ be an orthonormal frame of $\cE^q(\lambda,\square_k^E)$.
Let 
\begin{equation*}
    B^q_{\leq\lambda}(x):=\sum_{j}|s_j(x)|^2_h, \;{x}\in {X}
\end{equation*}
be the \textbf{Bergman kernel function for lower energy forms}. Let
\begin{equation*}
    {S}^{q}_{\leq \lambda}({x}):=\sup_{\alpha\in \cE^q(\lambda,\square_k)}\frac{|\alpha(x)|^{2}_h}{||\alpha||^{2}},
\;{x}\in {X}
\end{equation*}
be the \textbf{extremal function for lower energy forms}.
We denote the spectrum counting function of $\square_k$ by
$N^q(\lambda,\square_k):=\dim \cE^q(\lambda,\square_k).$
By the spectral theorem,
\begin{equation*}
\cH^{0,q}(X,L^k)=\cE^q(0,\square_k)\subset\cE^q(\lambda,\square_k)\subset \Dom(\square_k)\cap L^2_{0,q}(X,L^k).
\end{equation*}

We can also define component versions of $S_{k}^q$ and $S_{\leq\lambda}^q$. For a given orthonormal frame $e_{x}^{I}$ in $\Lambda^{0,q}_{x}(X,L^{k})$, let
\begin{equation*}
    {S}^{q}_{k,I}({x}):=\sup_{\alpha\in \cH^{q}(X,L^{k})}\frac{|\alpha_{I}(x)|^{2}_h}{||\alpha||^{2}},\;{x}\in {X}
\end{equation*}
be the $I$- component of ${S}^{q}_k({x})$, where $\alpha_{I}(x)$ denotes the component of $\alpha$ along $e^{I}_x$. And let
\begin{equation*}
    {S}^{q}_{\leq\lambda,I}({x}):=\sup_{\alpha\in \cE^q(\lambda,\square_k)}\frac{|\alpha_{I}(x)|^{2}_h}{||\alpha||^{2}},\;{x}\in {X}
\end{equation*}
be the $I$- component of ${S}^{q}_{\leq\lambda}({x})$.

For computing Bergman kernel function, we need the relation as follows.
\begin{lemma}[{\cite[Lemma 2.1]{Be04}}]\label{lebs}
    With notation as above, we have
    \begin{equation*}
        {S}^{q}_k({x})\leq {B}^{q}_k({x})\leq \sum_{I}^{'}S^{q}_{k,I}(x),\quad \forall x\in X.
    \end{equation*}
    We also have the relation between Bergman kernel function and extremal function for lower energy form
    \begin{equation*}
         {S}^{q}_{\leq \lambda}({x})\leq  B^q_{\leq\lambda}(x)\leq \sum_{I}^{'}S^{q}_{\leq\lambda,I}({x}),\quad \forall x\in X,
    \end{equation*}
   where the summations above are summing for strictly increasing index $I$. 
\end{lemma}

\subsection{The scaling technique}

In this subsection, we will recall the scaling technique in \cite{Be04}. Let $(X,\omega)$ be a Hermitian manifold of dimension $n$. Let $(L,h^L)$ be a holomorphic Hermitian line bundle on $X$. Fix $x\in X$, we can take a local complex coordinate $\{z_{i}\}$ around $x$ and a holomorphic trivializing section $s$ of $L$ such that\cite{We1}
\begin{equation*}
    \omega(z)= \frac{i}{2}\sum_{i,j}h_{ij}(z)dz_{i}\wedge\ov{dz_{j}},\quad h_{ij}(0)=\delta_{ij},
\end{equation*}
\begin{equation*}
    |s(z)|^{2}=e^{-\phi(z)},\quad \phi(z)=\sum_{i=1}^{n}\lambda_{i,x}|z_{i}|^{2}+O(|z|^{3}).
\end{equation*}
It is not difficult to check that 
\begin{equation}\label{eqdet}
    \prod_{i=1}^{n}(\lambda_{i,x})dV_{X,\omega}(x)=\frac{1}{n!}(\frac{i}{2}\dbar\ddbar\phi)_x^n=det_{\omega}(\frac{i}{2}\dbar\ddbar\phi)_{x}dV_{X,\omega}(x).
\end{equation}

Let's define 
$$\phi_{0}(z):=\sum_{i=1}^{n}\lambda_{i,x}|z_{i}|^{2}.
$$

We will introduce some scaling notation. Let $B_{R}:=\{z:|z|<R\}$ in $\mathbb{C}^n$ and let $R_k:=\frac{\log k}{\sqrt{k}}$. Under the local coordinate around $x$, $B_R$ is identified with a subset of $X$ for some $R<<1$. Given a function $f$ on the ball $B_{R_{k}}$, we define the scaled function of $f$ by
$$f^{(k)}: B_{\sqrt{k}R_{k}}=B_{\log k}\longrightarrow \mathbb{C},\quad z\mapsto f(\frac{z}{\sqrt{k}}).
$$
Differential forms are scaled by scaling the components. We can compute that scaling the fiber metric on $L^k$ gives 
\begin{equation}\label{exphi}
    (k\phi)^{(k)}(z)=\sum_{i=1}^n \lambda_{i} |z_{i}|^2 + \frac{1}{\sqrt{k}}O(|z|^3).
\end{equation}
The radius $R_{k}:=\frac{\log k}{\sqrt{k}}$ has been chosen to make sure that the fiber metric on $L^k$ tends to the model fiber metric $\phi_0$ with all derivatives on scaled balls
\begin{equation}\label{scaltomod}
    \sup_{|z|\leq\sqrt{k}R_k}\left|\partial^\alpha((k\phi)^{(k)}-\phi_0)(z)\right|\to0,
\end{equation}
since (\ref{exphi}) and for any $n\in \mathbb{N}$,
$$ \sup_{|z|\leq\sqrt{k}R_k}\frac{1}{\sqrt{k}}O(|z|^n)\leq C \frac{(\log k)^n}{\sqrt{k}}\to 0.$$
Moreover, $\sqrt{k}R_{k}=\log k$ tends to infinity, so that the sequence of scaled balls $B_{\sqrt{k}R_{k}}$ exhausts $\mathbb{C}^n$. Let's denote by $\square^{(k)}$ the Laplacian, taken with respect to the scaled fiber metric $(k\phi)^{(k)}$ and the scaled base metric $\omega^{(k)}$. One can check that
\begin{equation}\label{scalap}
    \square^{(k)}\alpha^{(k)}=\frac{1}{k}(\square_{k}\alpha_{k})^{(k)},\quad \forall 
  \alpha_{k}\in \Dom(\square_{k})\subset L^{2}_{0,q}(X,L^{k}),
\end{equation}
where we use the notation $\alpha^{(k)}$ to replace $\alpha^{(k)}_k$ for simplifying the notation. Therefore, for any $\alpha_k \in \cH^{q}(X,L^{k})$, then the scaled form $\alpha{(k)}$ satisfies 
$$\square^{(k)}\alpha^{(k)}=0
$$
on the scaled ball $B_{\sqrt{k}R_{k}}$. Furthermore, by (\ref{scaltomod}), it's not hard to check that
\begin{equation}\label{exLa}
    \square^{(k)}=\square_{\phi_0}+\varepsilon_{k}\mathcal{D}_k,
\end{equation}
where $\square_{\phi_0}$ is the Laplacian with respect to the model metric $\phi_0$ and the scaled base metric $\omega^{(k)}$, $\mathcal{D}_k$ is a second order partial differential operator with bounded variable coefficients on the scaled ball $B_{\sqrt{k}R_{k}}$ and $\varepsilon_{k}$ is a sequence tending to zero. In fact, we can also check that all the derivatives of the coefficients of $\mathcal{D}_k$ are uniformly bounded. By changing of variables, one can also check that for any $\alpha_{k}\in  L^{2}_{0,q}(X,L^{k})$,
\begin{equation}\label{scanor}
    \|\alpha_{k}\|_{{B_{{R_{k}}}}}\sim k^{n}\left\|\alpha^{(k)}\right\|_{{\phi_{0},\sqrt{k}R_{k}}}.
\end{equation}

\subsection{The upper bound of $ B^q_{\leq\lambda}$}

In this subsection, we give a proof of {\cite[Proposition 5.1]{Be04}} for non-compact Hermitian manifolds ({Proposition \ref{upbd}}). In fact, we observe that the proof in {\cite[Proposition 5.1]{Be04}} is also valid in non-compact case and the methods used are similar to {\cite[Lemma 3.1, Theorem 3.2]{Be04}}. 

\begin{proposition}\label{upbd}
    Let $(X,\omega)$ be a Hermitian manifold of dimension $n$. Let $(L,h^L)$ be a holomorphic Hermitian line bundle on $X$. Assume that $\mu_k \to 0$, then for any $x\in X$, the following estimate holds
    \begin{equation*}
         B_{\leq\mu_kk}^{q}(x)\leq k^n1_{X(q)}\left|det_\omega(\frac {i}{2\pi}\partial\overline{\partial}\phi)_x\right|+o(k^n),
    \end{equation*} 
    where $X(q)$ is the subset of $X$ consisting of points on which the curvature of the holomorphic Hermitian line bundle $(L,h^L)$ has exactly $q$ negative eigenvalues and $n-q$ positive eigenvalues.
\end{proposition}

We adjust the proof of {\cite[Lemma 3.1]{Be04}} and get the following Lemma \ref{uple}, which is necessary to prove Proposition \ref{upbd}. 

\begin{lemma}\label{uple}
    For each $k$, suppose that $\beta^{(k)}$ is a smooth $q$-form on the ball $B_{\sqrt{k}R_k}$ such that $\beta^{(k)}=k^{-\frac{1}{2}}\alpha^{(k)}$, where $\alpha_k\in \cE^q(k\mu_{k},\square_k)$ has a unit norm. Identify $\beta^{(k)}$ with a form in $L^{2}_{\phi_0}(\mathbb{C}^n )$ by extending with zero. Then there is constant $C$ independent of $k$ such that
    \begin{equation*}
        \sup_{z\in B_1}\left|\beta^{(k)}(z)\right|_{\phi_0}^2\leq C\left\|\beta^{(k)}\right\|_{\phi_{0,}B_2}^2.
    \end{equation*}
    Moreover, there is a subsequence of $\{\beta^{(k)}\}$ which converges uniformly with all derivatives on any ball in $\mathbb{C}^n$ to a smooth form $\beta$, where $\beta$ is in $L^{2}_{\phi_0}(\mathbb{C}^n )$.
\end{lemma}
\begin{proof}
    Fix a ball $B_R$ in $\mathbb{C}^n$. By G\r{a}rding's inequality for the elliptic operator $(\square^{(k)})^m$, we have the following estimates for the Sobolev norm of $\beta^{(k)}$ on the ball $B_R$ with $2m$ derivatives
    \begin{equation*}
        \left\|\beta^{(k)}\right\|_{\phi_{0},B_{R},2m}^{2}\leq C_{R,k}\left(\left\|\beta^{(k)}\right\|_{{\phi_{0},B_{2R}}}^{2}+\left\|(\square^{(k)})^{m}\beta^{(k)}\right\|_{\phi_0, {B_{2R}}}^{2}\right),
    \end{equation*}
    for any positive integers $m$. 
Since $\square^{(k)}$ converges to $\square_{\phi_0}$ on $B_{2R}$, we find that $C_{R,k}$ is independent of $k$, therefore, for any $m\in \mathbb{N}$,
\begin{equation}\label{Gabe}
     \left\|\beta^{(k)}\right\|_{\phi_{0},B_{R},2m}^{2}\leq C_{R}\left(\left\|\beta^{(k)}\right\|_{{\phi_{0},B_{2R}}}^{2}+\left\|(\square^{(k)})^{m}\beta^{(k)}\right\|_{\phi_0, {B_{2R}}}^{2}\right).
     \end{equation}
By changing of variables, we get
    \begin{equation*}
        \left\|(\square^{(k)})^{m}\beta^{(k)}\right\|_{\phi_{0},B_{2R}}^{2}\leq k^{-n}\left\|(\square^{(k)})^{m}\alpha^{(k)}\right\|_{{\phi_{0},B_{{\sqrt{k}R_{k}}}}}^{2}\lesssim k^{-2m}\left\|(\square_k)^{m}\alpha_{k}\right\|_{X}^{2}.
    \end{equation*}
    Since $\alpha_k\in \cE^q(k\mu_{k},\square_k)$ and $\alpha_k$ have unit norms, we have 
    \begin{equation*}
        k^{-2m}\left\|(\square_k)^{m}\alpha_{k}\right\|_{X}^{2}\leq \mu_{k}^m\left\|\alpha_{k}\right\|_{X}^{2}=\mu_k\to 0.
    \end{equation*}
Hence we find the last term $\left\|(\square^{(k)})^{m}\beta^{(k)}\right\|_{\phi_{0},B_{2R}}^{2}\to 0$ in (\ref{Gabe}). Therefore, we get 
\begin{equation}\label{Gabe2}
    \left\|\beta^{(k)}\right\|_{\phi_{0},B_{R},2m}^{2}\leq C_{R}\left\|\beta^{(k)}\right\|_{{\phi_{0},B_{2R}}}^{2},
\end{equation}
in particular,
$\left\|\beta^{(k)}\right\|_{\phi_{0},B_1,2m}^2\leq C\left\|\beta^{(k)}\right\|_{\phi_{0},B_2}^2.$
The continuous injection $L^{2,k}\hookrightarrow C^0, k>n$, provided by the Sobolev embedding theorem, proves the first statement in the lemma. Next, we will prove the second statement, since $\alpha_k$ have unit norms and (\ref{scanor}), we get
 \begin{equation*}
     \sup_k\left\|\beta^{(k)}\right\|_{\phi_0}^2=\sup_kk^{-n}\left\|\alpha^{(k)}\right\|_{\phi_0,B_{\sqrt{k}R_k}}^2\leq\sup_k\|\alpha_k\|_X^2=1.
 \end{equation*}
 By (\ref{Gabe2}), we deduce that for any $R>0$,
 \begin{equation*}
      \left\|\beta^{(k)}\right\|_{\phi_{0},B_{R},2m}^{2}\leq C_{R}.
 \end{equation*}
 Since this holds for any $m\geq1$, Rellich's compactness theorem yields, for each $R$, a subsequence of $\{\beta^{(k)}\}$ , which converges in all Sobolev spaces $\boldsymbol{H}^{k}(B_R)$  for $k \geq 0$. The compact embedding $\boldsymbol{H}^{k}\hookrightarrow C^l, k>n+\frac12l$, shows that the sequence converges in all $C^{l}(B_R)$. Choosing a diagonal sequence, with respect to a sequence
of balls exhausting $\mathbb{C}^n$, completes the proof of the lemma.
\end{proof}

Now we can give a proof of Proposition \ref{upbd}.
\begin{proof}[Proof of Proposition \ref{upbd}]
    First we will prove that 
    \begin{equation*}
        \limsup_{k\to\infty}k^{-n}S_{\leq\mu_kk}^{q}(x)\leq 1_{X(q)}\left|det_\omega(\frac {i}{2\pi}\partial\overline{\partial}\phi)_x\right|.
    \end{equation*}
    By definition, there is a sequence $\alpha_k\in \cE^q(k\mu_{k},\square_k)$ of unit norm such that 
    \begin{equation*}
         \limsup_{k\to\infty}k^{-n}S_{\leq\mu_kk}^{q}(x)=\limsup_{k\to\infty}k^{-n}|\alpha_k(x)|^2 .
    \end{equation*}
    Let's consider the $\beta^{(k)}:=k^{-\frac{1}{2}}\alpha^{(k)}$. By definition, $\beta^{(k)}$ is a form on $B_{\sqrt{k}R_k}$ and identify $\beta^{(k)}$ with a form in $L^{2}_{\phi_0}(\mathbb{C}^n )$ by extending with zero. Using Lemma \ref{uple}, we can find a subsequence of $\{\beta^{k_{j}}\}$ that converges uniformly with all derivatives to $\beta$ on any ball in $\mathbb{C}^n$, where $\beta$ is smooth and $||\beta||_{\phi_0}^2\leq1$. Since $\alpha_k\in \cE^q(k\mu_{k},\square_k)$, we get $\beta^{(k)}\in\cE^q(\mu_{k},\square^{(k)})$, moreover, from (\ref{exLa}), we get $\square_{\phi_0}\beta=0$. Hence, we get
    \begin{equation*}
        \lim_k\sup k^{-n}S^{q}_{\leq\mu_kk}(x)=\lim_j\left|\beta^{(k_j)}(0)\right|^2=|\beta(0)|^2\leq\frac{|\beta(0)|^2}{\|\beta\|_{\phi_0}^2}\leq S_{x,\mathbb{C}^n}^q(0),
    \end{equation*}
    Berman {\cite[Proposition 4.3]{Be04}} computed explicitly $S_{x,\mathbb{C}^n}^q(0)$, $B_{x,\mathbb{C}^n}^q(0)$ on model case $\mathbb{C}^n$
    \begin{equation*}
        S_{x,\mathbb{C}^n}^q(0)=B_{x,\mathbb{C}^n}^q(0)=1_{X(q)}(x)\left|det_\omega(\frac i{2\pi}\partial\overline{\partial}\phi)_x\right|.
    \end{equation*}
    The first step of this proof is done. By Lemma \ref{lebs}, we also get $\limsup_{k\to\infty}k^{-n}B_{\leq\mu_kk}^{q}(x)=0$, for any $x\notin X(q)$. Next, let's consider $x\in X(q)$, we can assume that $\lambda_1 \cdots \lambda_q$ are negative eigenvalues. Berman {\cite[Proposition 4.3]{Be04}} proved 
    \begin{equation*}
        S_{I,x,\mathbb{C}^n}^q(0)=0, \quad \forall I\ne (1,2,\dots,q).
    \end{equation*}
    Therefore, we have $\beta_I=0$ for any $I\ne (1,2,\dots,q)$. Moreover, for any $I\ne (1,2,\dots,q)$, we deduce that 
    \begin{equation*}
        \lim_j k_j^{-n}S_{\leq\lambda,I}^{q}(0)=\lim_j k_j^{-n}\left|\alpha_{k_j,I}(0)\right|^2=\left|\beta_I(0)\right|^2=0.
    \end{equation*}
    This proves that  
    \begin{equation*}
        \lim_k k^{-n}S_{\leq\lambda,I}^{q}(0)=0,\quad \forall I\ne (1,2,\dots,q).
    \end{equation*}
    Finally, by Lemma \ref{lebs},
    \begin{equation*}
        \lim\sup_kk^{-n}B_{\leq\mu_kk}^{q}(x)\leq0+0+...+S_{x,\mathbb{C}^n}^q(0)=1_{X(q)}(x)\left|det_\omega(\frac i{2\pi}\partial\overline{\partial}\phi)_x\right|
    \end{equation*}
    finishes the proof of the theorem.
\end{proof}

 \subsection{The lower bound of $ B^q_{\leq\lambda}$}
 In this subsection, we will prove the following lower bound of $ B^q_{\leq\lambda}$ on any compact subset of Hermitian manifolds (whether compact or not). 
\begin{proposition}\label{lobd}
    Let $(X,\omega)$ be a Hermitian manifold of dimension $n$. Let $(L,h^L)$ be a holomorphic Hermitian line bundle on $X$. Suppose $K$ is a compact subset of $X$. Then there is a sequence $\mu_k\to0$ such that 
    \begin{equation*}
        \liminf_{k\to\infty} k^{-n}B_{\leq\mu_kk}^{q}(x)\geq1_{K(q)}\left|det_\omega(\frac i{2\pi}\partial\overline{\partial}\phi)_x\right|,\qquad \forall x\in K.
    \end{equation*}
\end{proposition}

The following Lemma is needed when we prove Proposition \ref{lobd}. The proof comes from modifying {\cite[Lemma 5.2]{Be04}}. Then we extend the results to any compact subsets of non-compact Hermitian manifolds. 

\begin{lemma}\label{lelb}
    For any $x\in X(q)$, there is a sequence $\{\alpha_k\}\subset \Omega^{0,q}_0(X,L^k)$ such that
    \begin{equation*}
        \begin{split}
            &|\alpha_k(x)|^2=k^n\left|det_\omega(\frac i{2\pi}\partial\overline{\partial}\phi)_x\right|,\\
            &\lim_k\|\alpha_k\|_X^2=1,\\
            &\lim_k\|k^{-m}(\square_k)^m \alpha_k\|^2_X=0,\quad \forall m\in\mathbb{N}.
        \end{split}
    \end{equation*}
    Moreover, for any compact set $K$ of $X$, there is a sequence $\delta_k\to0$, such that 
    \begin{equation*}
        \left( k^{-1}\square_k\alpha_k,\alpha_k\right)_X\leq\delta_k,\quad \forall x\in K(q), \ k\in \mathbb{N}.
    \end{equation*}
\end{lemma}

\begin{proof}
    Without loss of generality, we assume that the first $q$ eigenvalues at $x$, $\lambda_{1,x},\dots\lambda_{q,x}$ are negative, while the remaining eigenvalues are positive. Define the following form in $\mathbb{C}^n$:
    \begin{equation*}
    \beta(w)=\left(\frac{|\lambda_1||\lambda_2|\cdotp\cdotp\cdotp|\lambda_n|}{\pi^n}\right)^{\frac12}e^{+\sum_{i=1}^q\lambda_i|w_i|^2}d\overline{w_1}\wedge d\overline{w_2}\wedge...\wedge d\overline{w_q}.
    \end{equation*}
    Observe that $|\beta|_{\phi_{0}}^{2}=\frac{|\lambda_{1}||\lambda_{2}|\cdots|\lambda_{n}|}{\pi^{n}}e^{-\sum_{i=1}^{n}|\lambda_{i}||w_{i}|^{2}}$ and $\|\beta\|_{\phi_0,\mathbb{C}^n}=1$. It is not hard to check that $\beta\in\boldsymbol{H}^{m}_{\phi_0}(\mathbb{C}^n)$ for any $m\in\mathbb{N}$. Define $\alpha_k$ on $X$ by 
    \begin{equation*}
        \alpha_k(z):=k^{\frac n2}\chi_k(\sqrt{k}z)\beta(\sqrt{k}z),
    \end{equation*}
    where $\chi_k(w)=\chi(\frac{w}{\sqrt{k}R_k})$ and $\chi$ is a smooth function supported on the unit ball, which equals one on the ball of radius $\frac{1}{2}$. By (\ref{eqdet}), it is easy to find that 
    \begin{equation*}
        |\alpha_k(x)|^2=k^n\left|det_\omega(\frac i{2\pi}\partial\overline{\partial}\phi)_x\right|.
    \end{equation*}
   To compute $\|\alpha_k\|^2$, note that 
   \begin{equation*} \|\alpha_k\|_X^2=\|\chi_k\beta\|_{\phi_0,\mathbb{C}^n}^2=\|\beta\|_{\phi_0,\frac12\sqrt{k}R_k}^2+\|\chi_k\beta\|_{\phi_0,\geq\frac12\sqrt{k}R_k}^2.
   \end{equation*}
   Since $\beta\in L^2_{\phi_0}(\mathbb{C}^n)$, we get the last term $\|\chi_k\beta\|_{\phi_0,\geq\frac12\sqrt{k}R_k}^2\to 0$. And by $\sqrt{k}R_k$ tends to infinity, the first term $\|\beta\|_{\phi_0,\frac12\sqrt{k}R_k}^2\to \|\beta\|_{\phi_0,\mathbb{C}^n}=1$. Then combining above, we get 
   \begin{equation*}
       \lim_k\|\alpha_k\|^2=1.
   \end{equation*}
   From (\ref{scalap}) and (\ref{exLa}), we get
   \begin{equation}\label{ex35}
       \begin{split}
          \left\|k^{-m}(\square_k)^m\alpha_k\right\|_X^2
          \lesssim& k^{-n}\left\|(\square^{(k)})^m\alpha^{(k)}\right\|_{\phi_0,\sqrt{k}R_k}^2
          =k^{-n}\left\|(\square^{(k)})^{m}k^{\frac{n}{2}}(\chi_k\beta))\right\|_{\phi_0,\sqrt{k}R_k}^2\\
          =&\left\|(\square^{(k)})^{m}(\chi_k\beta))\right\|_{\phi_0,\sqrt{k}R_k}^2
          =\left\|(\square^{(k)})^{m-1}(\square_{\phi_0}+\varepsilon_k\mathcal{D}_k)(\chi_k\beta))\right\|_{\phi_0,\sqrt{k}R_k}^2\\
          \leq &\left\|(\square^{(k)})^{m-1}\square_{\phi_0}(\chi_k\beta))\right\|_{\phi_0,\sqrt{k}R_k}^2+\varepsilon_k^2 \left\|(\square^{(k)})^{m-1}\mathcal{D}_k(\chi_k\beta))\right\|_{\phi_0,\sqrt{k}R_k}^2,
       \end{split}
   \end{equation}
   where $\mathcal{D}_k$ is a second order partial differential operator, whose coefficients have derivatives that are uniformly bounded in $k$ and $\varepsilon_k\to 0$. Note that the first term $\left\|(\square^{(k)})^{m-1}\square_{\phi_0}(\chi_k\beta))\right\|_{\phi_0,\sqrt{k}R_k}^2$ tends to $0$ as $k\to \infty$. Indeed, it is easy to check $\square_{\phi_0}\beta=0$, hence $\ddbar\beta=\ddbar^{*,\phi_0}\beta=0$. By Leibniz' rule
   \begin{equation*}
       \square_{\phi_{0}}(\chi_{k}\beta)=\eta_{k}\beta,
   \end{equation*}
   where $\eta_k$ is a function, uniformly bounded in $k$ and contains second derivatives of $\chi_k$. It is not hard to find $\supp\eta_k\subset B_{\sqrt{k}R_k} \setminus B_{\frac{1}{2}\sqrt{k}R_k}$. Using (\ref{exLa}) repeatedly, we find 
   \begin{equation}\label{ft35}
       \left\|(\square^{(k)})^{m-1}\square_{\phi_0}(\chi_k\beta))\right\|_{\phi_0,\sqrt{k}R_k}^2\lesssim\left\|\psi_k P\beta\right\|_{\phi_0,\mathbb{C}^n}^2, 
   \end{equation}
   where $\psi_k$ is a function and $\supp\psi_k\subset B_{\sqrt{k}R_k} \setminus B_{\frac{1}{2}\sqrt{k}R_k}$, $P$ is a polynomial. Indeed, coefficients of $\mathcal{D}_k$ and its derivatives are uniformly bounded in $k$ and if we take any derivatives to $\beta$, it should be polynomials multiply with $\beta$. Since $\supp\psi_k\subset B_{\sqrt{k}R_k} \setminus B_{\frac{1}{2}\sqrt{k}R_k}$ and $P\beta\in  L^2_{\phi_0}(\mathbb{C}^n)$, we get $\left\|\psi_k P\beta\right\|_{\phi_0,\mathbb{C}^n}^2\to0$, thus 
   \begin{equation*}
       \left\|(\square^{(k)})^{m-1}\square_{\phi_0}(\chi_k\beta))\right\|_{\phi_0,\sqrt{k}R_k}^2\longrightarrow0.
   \end{equation*}
   Use the same method as (\ref{ft35}), we get a estimate of last term of (\ref{ex35})
   \begin{equation*}
        \left\|(\square^{(k)})^{m-1}\mathcal{D}_k(\chi_k\beta))\right\|_{\phi_0,\sqrt{k}R_k}^2\lesssim\left\| Q\beta\right\|_{\phi_0,\mathbb{C}^n}^2,
   \end{equation*}
   where $Q$ is a polynomial. Since $Q\beta\in  L^2_{\phi_0}(\mathbb{C}^n)$, we get $\left\|(\square^{(k)})^{m-1}\mathcal{D}_k(\chi_k\beta))\right\|_{\phi_0,\sqrt{k}R_k}^2$ is uniformly bounded. Combining (\ref{ex35}),(\ref{ft35}), we get
   \begin{equation*}
       \lim_k\|k^{-m}(\square_k)^m \alpha_k\|^2_X=0,\quad \forall m\in\mathbb{N}.
   \end{equation*}
   To prove the last statement, observe that
   \begin{equation*}
   \begin{split}
      \left( k^{-1}\square_k \alpha_k,\alpha_k\right)_X
      =&\left\|\frac1{\sqrt{k}}(\overline{\partial}+\overline{\partial}^*)\alpha_k\right\|_X^2
      \sim\left\|(\overline{\partial}+\overline{\partial}^{*(k)})(\chi_{\sqrt{k}R_k}\beta)\right\|_{\sqrt{k}R_k}^2\\
      \lesssim&\left\|(\chi_{\sqrt{k}R_k}(\overline{\partial}+\overline{\partial}^{*(k)})\beta\right\|_{\sqrt{k}R_k}^2+\|\beta\|_{ (B_{\sqrt{k}R_k} \setminus B_{\frac{1}{2}\sqrt{k}R_k})}^2\\
      \lesssim &\varepsilon_k\left(\|\beta\|^2+\sum_{i=1}^{2n}\|\partial_i\beta\|^2\right)+\|\beta\|_{ (B_{\sqrt{k}R_k} \setminus B_{\frac{1}{2}\sqrt{k}R_k})}^2,
   \end{split}    
   \end{equation*}
   where $\varepsilon_k\to0$, the third inequality comes from Leibniz' rule and the last inequality holds since there is an expansion for the first order operator $(\overline{\partial}+\overline{\partial}^{*,(k)})$ as in (\ref{exLa}) and $(\overline{\partial}+\overline{\partial}^{*,\phi_0})\beta=0$. Observe that $\dbar_i\beta$ depends on the eigenvalues of $\dbar\ddbar\phi$ at $x$, hence it depends on $x$. But for any compact set $K$ of $X$, there is an upper bound of absolute value of eigenvalues of $\dbar\ddbar\phi$ on $K$, we also have $\|\beta\|^2=1$, therefore $\left(\|\beta\|^2+\sum_{i=1}^{2n}\|\partial_i\beta\|^2\right)<1+C$. By the same argument, we can deduce $\varepsilon_k$ is independent of $x$. Since $\|\beta\|_{ (B_{\sqrt{k}R_k} \setminus B_{\frac{1}{2}\sqrt{k}R_k})}^2\to0$, we can find $\delta_k\to0$, such that for any $k$,
   \begin{equation*}
       \left( k^{-1}\square_k \alpha_k,\alpha_k\right)_X\leq\delta_k,\quad \forall x\in K(q), \ k\in \mathbb{N}.
   \end{equation*}
   The proof of this lemma is completed.
\end{proof}

Now we can prove Proposition \ref{lobd}.
\begin{proof}[Proof of Proposition \ref{lobd}]
    Take $\mu_k\to0$, such that $\frac{\delta_k}{\mu_k}\to0$, where $\delta_k$ is the sequence in Lemma \ref{lelb}.
    Let $\{\alpha_k\}$ be the sequence that Lemma \ref{lelb} provides. Define
    \begin{equation*}
        \alpha_{1,k}:= E_{\leq\mu_kk}^q\alpha_k,\quad \alpha_{1,k}:=\alpha_k-\alpha_{1,k},
    \end{equation*}
    where $E_{\leq\mu_kk}^q$ is the the spectral projection of $\square_k$. By the definition, $\alpha_k\in\Omega^{0,q}_0(X,L^{k})\subset L^2(X,L^k)$. Hence $\alpha_{1,k}\in\cE^q(\mu_kk,\square_k)$.    By the elliptic property of $\square_k$, $\ \alpha_{1,k},\alpha_{2,k}\in \Omega^{0,q}(X,L^{k})\cap L^2(X,L^k)$. Moreover, we can check
    \begin{equation}\label{proj0}
        \left(E^q_{\leq\lambda}\alpha_{2,k},\alpha_{2,k}\right)_X=0, \quad\forall\lambda\leq\mu_kk. 
    \end{equation}
   
    First, we prove the claim
    \begin{equation}\label{cllb}
        \lim_kk^{-n}\left|\alpha_2^{(k)}(0)\right|^2=0.
    \end{equation}
    As in the proof of Lemma \ref{uple}, 
    \begin{equation}\label{exlb}
        k^{-n}\left|\alpha_2^{(k)}(0)\right|^2
        \leq C\left(k^{-n}\left\|\alpha_2^{(k)}\right\|_{\phi_0,B_1}^2+k^{-n}\left\|(\square^{(k)})^m\alpha_2^{(k)}\right\|_{\phi_0,B_1}^2\right).
    \end{equation}
    From (\ref{proj0}), we observe that
     \begin{equation*}  
     \begin{split}
          \left(\square_k\alpha_{2,k},\alpha_{2,k}\right)_X
        =&\left(\int_\mathbb{R}\lambda 
 dE^q_{\leq\lambda}\alpha_{2,k},\alpha_{2,k}\right)_X
 =\int_\mathbb{R}\lambda 
 d\left(E^q_{\leq\lambda}\alpha_{2,k},\alpha_{2,k}\right)_X  \\
 =&\int_{(\mu_kk,+\infty]}\lambda 
 d\left(E^q_{\leq\lambda}\alpha_{2,k},\alpha_{2,k}\right)_X
 \geq \mu_kk\int_ \mathbb{R}1
 d\left(E^q_{\leq\lambda}\alpha_{2,k},\alpha_{2,k}\right)_X\\
 =&\mu_kk\left(\int_ \mathbb{R}1
 dE^q_{\leq\lambda}\alpha_{2,k},\alpha_{2,k}\right)_X
 =\mu_kk\|\alpha_{2,k}\|_X^2.
     \end{split}   
    \end{equation*}
    Hence, by Lemma \ref{lelb}, we find the first term of (\ref{exlb})
    \begin{equation*}
        k^{-n}\left\|\alpha_2^{(k)}\right\|_{\phi_0,B_1}^2
        \lesssim\left\|\alpha_{2,k}\right\|_{X}^{2}
        \leq\frac{1}{\mu_{k}k}\left(\square_k\alpha_{2,k},\alpha_{2,k}\right)_{X}
        \leq\frac{1}{\mu_{k}}\left( k^{-1}\square_k\alpha_{k},\alpha_{k}\right)_{X}\leq\frac{\delta_{k}}{\mu_{k}}\to 0.
    \end{equation*}
    Lemma \ref{lelb} also tells us the second term of (\ref{exlb}) tends to zero,
    \begin{equation*}
        k^{-n}\left\|(\square^{(k)})^m\alpha_2^{(k)}\right\|_{\phi_0,B_1}^2\leq\left\|k^{-m}(\square_k)^m\alpha_{2,k}\right\|_X^2\leq\left\|k^{-m}(\square_k)^m\alpha_k\right\|_X^2\to0.
    \end{equation*}
    Now we finish the proof of claim (\ref{cllb}).

    Finally, combining Lemma \ref{lebs}, Lemma \ref{lelb} and claim (\ref{cllb}), for any $x\in K(q)$ 
    \begin{equation*}
    \begin{split}
        k^{-n}B_{\mu_kk}^{q}(x)
        \geq& k^{-n}S_{\mu_kk}^{q}(x)
        \geq k^{-n}\frac{\left|\alpha_{k,1}(0)\right|^2}{\left\|\alpha_{k,1}\right\|_X^2}\geq k^{-n}\left|\alpha_{1,k}(0)\right|^2\\
        =& k^{-n}\left|\alpha_k(0)-\alpha_{2,k}(0)\right|^2
        \longrightarrow k^{-n}\left|\alpha_k(0)\right|^2
        =\left|det_\omega(\frac i{2\pi}\partial\overline{\partial}\phi)_x\right|.
    \end{split}
    \end{equation*}
    Since $B_{\mu_kk}^{q}(x)=0$ for any $x\notin K(q)$, we finish the proof.
\end{proof}

At the end of this section, we give a local version of strong holomorphic Morse inequalities on any compact subsets of non-compact Hermitian manifolds. In fact, it can be directly deduced  from Proposition \ref{upbd} and Proposition \ref{lobd}.

\begin{theorem}\label{thul}
    Let $(X,\omega)$ be a Hermitian manifold of dimension $n$. Let $(L,h^L)$ be a holomorphic Hermitian line bundle on $X$. Suppose $K$ is a compact subset of $X$. Then there is a sequence $\mu_k\to0$ such that 
    \begin{equation*}
     \lim_{k\to\infty} k^{-n}B_{\leq\mu_kk}^{q}(x)=1_{K(q)}\left|det_\omega(\frac i{2\pi}\partial\overline{\partial}\phi)_x\right|,\qquad \forall x\in K.
    \end{equation*}
\end{theorem}
\begin{proof}
    Take $\mu_k:=\sqrt{\delta_k}$, where $\delta_k$ is the sequence in Lemma \ref{lelb}. Then the theorem comes from Proposition \ref{upbd} and Proposition \ref{lobd} directly.
\end{proof}

\section{Strong Holomorphic Morse inequalities}\label{sec_HMI}

In this section, we will prove $L^2$ strong holomorphic Morse inequalities on non-compact manifolds (Theorem \ref{prop_main}). The main tools we used are based on Section \ref{sec_Asy}. First, we will give an asymptotic property of dimension of lower energy form Bergman spaces $N^q(\mu_kk,\square_k)=\dim_{\mathbb{C}} \cE^q(\mu_kk,\square_k).$ The point of passing from the compact manifold to the non-compact is, under appropriate assumption, the norm of lower energy forms with values in $L^k$ decay to zero as $k\rightarrow\infty$ outside of a compact subset. As a consequence, the computation of $N^q(\mu_kk,\square_k)$ concentrates on a compact subset.

\begin{theorem}\label{esdeN}
    Let $(X,\omega)$ be a Hermitian manifold of dimension $n$ and let $(L,h^L)$ be a holomorphic Hermitian line bundle on $X$. Let $0\leq q\leq n$.
		Suppose there exist a compact subset $K\subset X$ and $C>0$ such that,
		for sufficiently large $k$, we have
		\begin{equation*}
		\left(1-\frac{C}{k}\right)||s||^2\leq \frac{C}{k}\left(||\ddbar_ks||^2+||\ddbar^{*}_{k}s||^2\right)+\int_{K} |s|^2 dV_X
		\end{equation*}
		for $s\in \Dom(\ddbar_k)\cap \Dom(\ddbar^{*}_{k})\cap L^2_{0,q}(X,L^k)$.
	Then  we have the estimate for the dimension of the $q$-th lower energy form Bergman spaces $N^q(\mu_kk,\square_k)$,
	\begin{equation*}
	\lim_{k\rightarrow \infty}n!k^{-n}N^q(\mu_kk,\square_k)= \int_{K(q)}(-1)^q c_1(L,h^L)^n.
	\end{equation*}
\end{theorem}
Note that the condition in Theorem \ref{esdeN} is the $q$-th optimal fundamental estimate.
\begin{proof}
    First, by the spectral decomposition theorem, using the optimal fundamental estimate, for any $s\in \cE^q(\mu_kk,\square_{k})\subset \Dom(\square_{k})\cap L^2_{0,q}(X,L^k)$,
	\begin{equation*}
		\begin{split}
			(1-\frac{C}{k})||s||^2\leq& \frac{C}{k}(||\ddbar^{k}s||^2+||\ddbar^{*}_{k}s||^2)+\int_{K} |s|^2 dV_X\\
			=&\frac{C}{k}(\square^{k}s,s)+\int_{K} |s|^2 dV_X
			=\frac{C}{k}\left(\left(\int_{\mathbb{R}}\lambda dE_{\leq\lambda}^{q}\right)s,s\right)+\int_{K} |s|^2 dV_X\\
			=&\frac{C}{k}\int_{[0,\mu_kk]}\lambda d(E_{\leq\lambda}^{q}s,s)+\int_{K} |s|^2  dV_X
			\le \frac{C}{k} \mu_kk\int_{\mathbb{R}}1 d(E_{\leq\lambda}^{q}s,s)+\int_{K} |s|^2  dV_X\\
			=&C\mu_k\left(\left(\int_{\mathbb{R}}1dE_{\leq\lambda}^{q}\right)s,s\right)+\int_{K} |s|^2  dV_X
			=C\mu_k\|s\|^2+\int_{K} |s|^2 dV_X,	
		\end{split}		
	\end{equation*}
	where $E_{\leq\lambda}^{q}$ is the spectral measure of $\square_{k}$.
	Thus, it follows that $\|s\|^2\leq c_k\int_{K} |s|^2 dV_X$, where $c_k:=\frac{k}{k-\mu_kk-C}$ and $c_k \to 1$, as $k\to \infty$.
	By Fatou's lemma, H\"{o}lder's inequality and Theorem\ref{thul}, we get
	\begin{equation}\label{upbdN}
		\begin{split}
			&\limsup_{k\rightarrow \infty}\left(k^{-n}N^q(\mu_kk,\square_k)\right)\\
			\leq&\limsup_{k\rightarrow \infty}\left(k^{-n}c_k \int_{K}B_{\leq \mu_kk}^q(x)dV_X(x)\right)\\
			\leq&\left(\limsup_{k\rightarrow \infty} c_k\right)\left(\limsup_{k\rightarrow \infty} \int_{K}k^{-n}B_{\leq \mu_kk}^q(x)dV_X(x)\right)\\
			\leq&\int_{K}\lim_{k\rightarrow \infty}k^{-n} B_{\leq \mu_kk}^q(x)dV_X(x)
                =\int_{K}1_{K(q)}\left|det_\omega(\frac i{2\pi}\partial\overline{\partial}\phi)_x\right|dV_X(x)\\
			=& \int_{K(q)}(-1)^q \frac{c_1({L},h^{{L}})^n}{n!}.
		\end{split}
	\end{equation}
 Therefore, we get the upper bound of $k^{-n}N^q(\mu_kk,\square_k)$. Note that the idea of estimating the upper bound of $k^{-n}N^q$ comes from {\cite[Proposition 4.2]{LSW}}. Next, let's consider the lower bound of $k^{-n}N^q(\mu_kk,\square_k)$, by Fatou's lemma and Theorem \ref{thul},
 \begin{equation}\label{lobdN}
     \begin{split}
         &\liminf_{k\rightarrow \infty}\left(k^{-n}N^q(\mu_kk,\square_k)\right)
         =\liminf_{k\rightarrow \infty}\left( k^{-n} \int_{X}B_{\leq \mu_kk}^q(x)dV_X(x)\right)\\
         \geq &\liminf_{k\rightarrow \infty}\left( k^{-n} \int_{K}B_{\leq \mu_kk}^q(x)dV_X(x)\right)
         \geq k^{-n} \int_{K}\lim_{k\rightarrow \infty}B_{\leq \mu_kk}^q(x)dV_X(x)\\
         =&\int_{K}1_{K(q)}\left|det_\omega(\frac i{2\pi}\partial\overline{\partial}\phi)_x\right|dV_X(x)
         =\int_{K(q)}(-1)^q \frac{c_1({L},h^{{L}})^n}{n!}.
     \end{split}
 \end{equation}
 Finally, combining (\ref{upbdN}) and (\ref{lobdN}), we finish the proof.
\end{proof}
We can use Theorem \ref{esdeN} to deduce the weak Morse inequalities (Corollary \ref{weakMorse}) directly.
\begin{proof}[Proof of Corollary \ref{weakMorse}]
    Using the canonical isomorphism of the weakly Hodge decomposition (\ref{weak H}) $\cH^{0,q}(M,E) \cong \overline{H}^{0,q}_{(2)}(X,L^k)$, and the fact that we have $\overline{H}^{0,q}_{(2)}(X,L^k)\cong H^{0,q}_{(2)}(X,L^k)$, when the fundamental estimate holds, we have $\dim H^{q}_{(2)}(X,L^{k}):=\dim H^{0,q}_{(2)}(X,L^{k})= \dim \cH^{0,q}(X,L^{k})\le N^q(\mu_kk,\square_k)$. This allows us to get the weak Morse inequalities directly by Theorem \ref{esdeN}.
\end{proof}

To prove Theorem \ref{prop_main}, we need the following Lemma \cite[Lemma 3.2.12]{MM07}.
\begin{lemma}\label{algmorse}
    Let
    \begin{equation*}
        0\longrightarrow V^0\xrightarrow{d^0}V^1\xrightarrow{d^1}\cdots\xrightarrow{d^{n-1}}V^n\longrightarrow0
    \end{equation*}
    be a complex of vector spaces. Let $H^i(V^\bullet)=\mathrm{Ker}(d^i)/\mathrm{Im}(d^{j-1})$ with $\mathrm{Im}(d^{-1})=0.$ If $\operatorname{dim}V^{q}<+\infty$ for any $q\le m$, then
    \begin{equation*}
        \sum_{j=0}^q(-1)^{q-j}\dim H^j(V^\bullet)\leqslant\sum_{j=0}^q(-1)^{q-j}\dim V^j, \quad\sum_{j=q}^n(-1)^{j-q}\dim H^j(V^\bullet)\leqslant\sum_{j=q}^n(-1)^{j-q}\dim V^j.   
    \end{equation*}
In particular, we also have
\begin{equation*}
      \sum_{j=0}^n(-1)^{j}\dim H^j(V^\bullet)=\sum_{j=0}^n(-1)^{j}\dim V^j.
\end{equation*}
\end{lemma}
Now, we are in the position of proving Theorem \ref{prop_main}.

\begin{proof}[Proof of Theorem \ref{prop_main}]
We only prove the first statement of Theorem \ref{prop_main}, the proof of other statements are same as the first.
Since $[\ddbar_k,\square_k]=0$, we get $\ddbar_k(\cE^q(\mu_kk,\square_k))\subset(\cE^{q+1}(\mu_kk,\square_k))$.
    Consider the chain complex
    \begin{equation*}
          0\longrightarrow \cE^0(\mu_kk,\square_k)\xrightarrow{\ddbar_k}\cE^1(\mu_kk,\square_k)\xrightarrow{\ddbar_k}\cdots\xrightarrow{\ddbar_k}\cE^n(\mu_kk,\square_k)\longrightarrow0.
    \end{equation*}
    Suppose $0\le q\le n$. If there is a compact set $K$ in $X$, such that $j$-th optimal fundamental estimates hold for any $0\le j\le q$. By Theorem \ref{esdeN}, we get
    \begin{equation*}
        \lim_{k\rightarrow \infty}n!k^{-n}N^j(\mu_kk,\square_k)= \int_{K(j)}(-1)^j c_1(L,h^L)^n, \quad \forall 0\le j\le q.
    \end{equation*}
    By the first statement of Lemma \ref{algmorse}, for any $0\le r\le q$, we get
    \begin{equation}\label{usealgM}
    \begin{split}
         &\lim_{k\to\infty} n!k^{-n}\sum_{j=0}^r(-1)^{r-j}\dim H^j(\cE^{\bullet}(\mu_kk,\square_k))
         \le n!k^{-n}\sum_{j=0}^r(-1)^{r-j}\lim_{k\to\infty}\dim \cE^j(\mu_kk,\square_k)\\  
         =& \sum_{j=0}^r(-1)^{r-j}\lim_{k\to\infty}N^j(\mu_kk,\square_k)
         \le \sum_{j=0}^r(-1)^{r-j}\int_{K(j)}(-1)^j c_1(L,h^L)^n
         =\int_{K(\le r)}(-1)^r c_1(L,h^L)^n.
    \end{split}     
   \end{equation}
   Next, we will prove that $H^j(\cE^{\bullet}(\mu_kk,\square_k))\cong H_{(2)}^j(X,L^k)$.
Since the fundamental estimates hold, using strong Hodge decomposition Theorem \ref{strong H}, for any $0\le j\le q$, we get 
\begin{equation*}
    \cE^j(\mu_kk,\square_k)=\cE^j(\mu_kk,\square_k)\cap L^2_{0,j}(X,L^k)=\cE^j(\mu_kk,\square_k)\cap(\cH^{0,j}(X,L^k)\oplus \Im^{j+1}(\ddbar^{*}_k)\oplus\Im^{j-1}(\ddbar_k)).
\end{equation*}
Since $[\ddbar_k,\square_k]=0$, we get 
\begin{equation*}
     \cE^j(\mu_kk,\square_k)=\cH^{0,j}(X,L^k)\oplus \Im^{j+1}(\ddbar^{*}_k|_{\cE^{\bullet}(\mu_kk,\square_k)})\oplus\Im^{j-1}(\ddbar_k|_{\cE^{\bullet}(\mu_kk,\square_k)}).
\end{equation*}
It's not hard to check 
\begin{equation*}
    \Ker^{j}(\cE^{\bullet}(\mu_kk,\square_k))=\cH^{0,j}(X,L^k)\oplus\Im^{j-1}(\ddbar_k|_{\cE^{\bullet}(\mu_kk,\square_k)}).
\end{equation*}
Hence, by strong Hodge decomposition Theorem \ref{strong H},
\begin{equation*}
    H^j(\cE^{\bullet}(\mu_kk,\square_k)):= \Ker^{j}(\cE^{\bullet}(\mu_kk,\square_k))/\Im^{j-1}(\ddbar_k|_{\cE^{\bullet}(\mu_kk,\square_k)})\cong \cH^{0,j}(X,L^k)\cong  H_{(2)}^j(X,L^k).
\end{equation*}
Combining with (\ref{usealgM}), we complete the proof of Theorem \ref{prop_main}.
   
\end{proof}

\section{Examples and Applications}\label{Sec_l2wmi}
In this section, we give some examples and prove Theorem \ref{thm_w1c}--Theorem \ref{thm_complete}. The main materials rely on Li-Shao-Wang \cite{LSW}, Peng-Shao-Wang \cite{PSW} and Ma-Marinescu\cite{MM07}.

\subsection{weakly $1$-complete manifolds}
In this subsection, we follow \cite{PSW} to prove Theorem \ref{thm_w1c}. The strong holomorphic Morse inequalities for weakly $1$-complete manifolds appeared in \cite{Bou:89,M:92,MM07}. In particular \cite{M:92} answered an open question of Ohsawa \cite{Oh:82} affirmatively. We give some necessary definitions firstly.

\begin{definition}
	A complex manifold $X$ is said to be weakly $1$-complete{ \cite{Nak:70}}, if there is a plurisubharmonic function $\varphi \in C^{\infty}(X, \mathbb{R})$, such that $X_c:=\{x \in X : \varphi(x) <c \} \Subset X$ for any $c \in\mathbb{R}$. 
	A Hermitian line bundle $(L,h^L)$ on a complex manifold $X$ is said to be Griffiths $q$-positive at $x \in X$, if the curvature form $R^L$ has at least $n-q+1$ positive eigenvalues at $x$, where $n =\operatorname{dim}_{\mathbb{C}} X$, $1 \leq q \leq n$. 
\end{definition}

 We suppose that $X$ is a weakly $1$-complete manifold of dimension $n$ and $\varphi$ is the exhaustion function of $X$. Let $(L,h^L)$ be a holomorphic Hermitian line bundle on $X$ and $K$ be a compact set of $X$. Assume $(L,h^L)$ is Griffiths $q$-positive on $X\setminus K$ with $q\ge 1$. Fix some $X_c=\{x \in X : \varphi(x) <c \}$ and assume $K\subset X_{c}\Subset X$. 

 Peng-Shao-Wang \cite{PSW} prove that optimal fundamental estimate holds for $X_c$ as follows.

\begin{proposition}[{\cite[Proposition 4.3]{PSW}}]\label{w1cofe}
	Let $X$ be a weakly $1$-complete  manifold of dimension $n$, $(L,h^L)$ be a Hermitian line bundle on complex manifold $X$, which is Griffiths $q$-positive on the outside of the compact set $K\subset X_c$ . Then there exist a compact subset $K'\subset X_c$ with $K\subset K'$ and $C>0$ such that for sufficiently large $k$, we deduce that
	\begin{equation*}
		(1-\frac{C}{k})||s||^2\leq \frac{C}{k}(||\ddbar_{k}s||^2+||\ddbar^{*}_{k,H} s||^2)+\int_{K'} |s|^2 dV_X
	\end{equation*}
	for any $s\in \Dom(\ddbar_{k})\cap \Dom(\ddbar^{*}_{k,H})\cap L^2_{0,j}(X_c,L^k)$ and $q\leq j \leq n$,
	where the $L^2$-norm is given by $\omega$, $h^{L^k}$ on $X_c$.	
	
\end{proposition}

Using the optimal fundamental estimate, we can give a proof of Theorem \ref{thm_w1c}.

    \begin{proof}[Proof of Theorem \ref{thm_w1c}]
	 By Proposition {\ref{w1cofe}} and Theorem {\ref{prop_main}}, we deduce that
	for any $q\le r\le n$,
	\begin{equation*}
		\sum_{j=r}^n(-1)^{j-r}\dim_\mathbb{C}H_{(2)}^j(X_c,L^k)\leq \frac{k^n}{n!}\int_{K^{'}(\ge r)}{(-1)^r c_1(L,h^L)^n}+o(k^n).
	\end{equation*}
	It only needs to prove that $K(j)=K^{'}(j)$ for $j\ge q$. Since $K\subset K^{'}$, we have $K(j)\subset K^{'}(j)$. For the opposite direction, since $L$ is Griffiths $q$- positive on $M\setminus K$,  $R^L$ has $n-q+1$ positive eigenvalues in $M\setminus K$ at least. So $R^{L}$ has $n-(n-q+1)=q-1<j$ negative eigenvalues at most. In particular, when $L>0$, by\cite[Theorem 6.2]{Takegoshi:83}, we have $ H^{j}(M,L^k)\cong \cH^{0,j}(X_c,L^{k})\cong  H^{0,j}_{(2)}(X_c,L^{k})$ for $k\gg1$ and every $i\in \mathbb{N}^+$. This completes the proof. 
\end{proof}

\subsection{Pseudoconvex domains}

	Let $M$ be a relatively compact domain with smooth boundary $bM$ in a complex manifold $X$. Let $\rho\in \cC^\infty(X,\R)$ such that $M=\{ x\in X: \rho(x)<0 \}$ and $d\rho\neq 0$ on $bM=\{x\in X: \rho(x)=0\}$. We denote the closure of $M$ by $\overline{M}=M\cup bM$. We say that $\rho$ is a defining function of $M$. Let $T^{(1,0)}bM:=\{ v\in T^{(1,0)}X: \dbar\rho(v)=0 \}$ be the analytic tangent bundle to $bM$. The Levi form of $\rho$ is the $2$-form $\cL_\rho:=\dbar\ddbar\rho\in \cC^\infty(bM, T^{(1,0)*}bM\otimes T^{(0,1)*}bM)$.	
	A relatively compact domain $M$ with smooth boundary $bM$ in a complex manifold $X$ is called pseudoconvex if the Levi form $\cL_\rho$ is semi-positive definite. 

   In \cite{LSW}, the authors give a proof of optimal fundamental estimate in pseudoconvex domain $M$. 
   \begin{proposition}[{\cite[Proof of Theorem 1.4]{LSW}}]\label{pse_ofe}
       Let $M\Subset X$ be a smooth pseudoconvex domain in a complex manifold $X$ of dimension $n$. Let $(L,h^L)$ be a holomorphic Hermitian line bundle on $X$.  Let $(L,h^L)$ be positive in a neighbourhood of the boundary $bM$ of $M$.
	Then there is a compact subset $K^{'}\Subset M$, a constant $C>0$ and a Hermitian metric $\omega$ on $X$, such that for sufficiently large $k$
      \begin{equation*}
		(1-\frac{C}{k})||s||^2\leq \frac{C}{k}(||\ddbar_{k}s||^2+||\ddbar^{*}_{k} s||^2)+\int_{K'} |s|^2 dV_X
	\end{equation*}
	for any $s\in \Dom(\ddbar_{k})\cap \Dom(\ddbar^{*}_{k})\cap L^2_{0,j}(M,L^k)$ and $1\leq j \leq n$,
	where the $L^2$-norm is given by $\omega$, $h^{L^k}$ on $M$.	
   \end{proposition}
Since optimal fundamental estimate holds on $M$ for any $(0,j)$-forms ($1\le j\le n$), using Theorem {\ref{prop_main}}, under the same conditions as in Proposition \ref{pse_ofe}, for any $1\le r\le n$, we get
 \begin{equation*}
     \sum_{j=r}^n(-1)^{j-r}\dim_\mathbb{C}H_{(2)}^j(M,L^k)\leq \frac{k^n}{n!}\int_{K^{'}(\ge r)}{(-1)^r c_1(L,h^L)^n}+o(k^n).
 \end{equation*}
 Let $K:=K^{'}$, we can get Theorem {\ref{thm_psc}} directly.

\subsection{$q$-convex manifolds}

A complex manifold $X$ of dimension $n$ is called $q$-convex (see \cite{AG:62}) if there exists a smooth function $\varrho\in \cC^\infty(X,\R)$ such that the sublevel set $X_c=\{ \varrho<c\}\Subset X$ for all $c\in \R$ and the complex Hessian $\dbar\ddbar\varrho$ has $n-q+1$ positive eigenvalues outside a compact subset $K\subset X$. Here $X_c\Subset X$ means that the closure $\overline{X}_c$ is compact in $X$. We call $\varrho$ an exhaustion function and $K$ exceptional set.

From now on let $X$ be a $q$-convex manifold of dimension $n$. Let $u_0<u<c<v$ such that the exceptional subset $K\subset X_{u_0}:=\{x\in X: \varrho(x)<{u_0} \}$. Then, we modify the prescribed hermitian metric $h^L$ on $L$.
Let $\chi(t)\in\cC^\infty(\R)$ such that $\chi'(t)\geq 0$, $\chi''(t)\geq 0$. We define a Hermitian metric $h^{L^k}_\chi:=h^{L^k}e^{-k\chi(\varrho)}$ on $L^k$ for each $k\geq 1$ and we set $L^k_\chi:=(L^k,h^{L^k}_\chi)$. Thus
\begin{equation*}
R^{L^k_\chi}=kR^{L_\chi}=kR^L+k\chi'(\varrho)\dbar\ddbar\varrho+k\chi''(\varrho)\dbar\varrho\wedge\ddbar\varrho.
\end{equation*}

Li-Shao-Wang{\cite{LSW}} prove the optimal fundamental estimate on $q$-convex manifolds as follows.
\begin{proposition}[{\cite[Proposition 3.4]{LSW}}]\label{1coxFE}
	Let $X$ be a $q$-convex manifold of dimension $n$ with the exceptional set $K\subset X_c$. Then there exist a compact subset $K'\subset X_c$ with $K\subset K'$, $C_0>0$ and $C_1>0$ such that for any sufficiently large $k$ and  $\chi(t)\in\cC^\infty(\R)$ satisfying $\chi'(\varrho)\geq C_1$ on $X_v\setminus \overline{X}_u$, we have
	\begin{equation*}
	(1-\frac{C_0}{k})||s||^2\leq \frac{C_0}{k}(||\ddbar_ks||^2+||\ddbar^{*}_{k,H}s||^2)+\int_{K'} |s|^2 dv_X
	\end{equation*}
	for any $s\in \Dom(\ddbar_k)\cap \Dom(\ddbar^{*}_{k,H})\cap L^2_{0,j}(X_c,L^k)$ and $q\leq j \leq n$,
	where the $L^2$-norm is given by $\omega$, $h^{L^k}_\chi$  on $X_c$.
\end{proposition}

Then using results above, we can give a proof of Theorem \ref{thm_q_convex}.
\begin{proof}[Proof of Theorem \ref{thm_q_convex}]
    Let $u_0<u<c<v$ such that $K\subset X_{u_0}\Subset K'\Subset X_c\Subset X_v$. We can suppose $K\cup M\subset X_{u_0}$ by choosing a suitable $u_0$, where $M$ is a compact set in the condition of Theorem \ref{thm_q_convex}. We choose now $\chi=\chi(t)\in\cC^\infty(\R)$, $\chi'(t)\geq 0$, $\chi''(t)\geq 0$ for all $t\in\R$ such that  $\chi=0$ on $(-\infty,u_0)$ and $\chi'(\varrho)\geq C_3>0$ on $X_v\setminus \ov X_u$. From Proposition \ref{1coxFE} and Theorem \ref{prop_main}, there exists a compact subset $K'\subset X_c$ with $K\subset K'$ such that
     for any $s+q-1\le r\le n$, we have 
  \begin{equation}\label{eq_h_xc}
     \sum_{j=r}^n(-1)^{j-r}\dim_\mathbb{C}H_{(2)}^j(X_c,L^k)\leq \frac{k^n}{n!}\int_{K'(\ge r,h_\chi^L)}{(-1)^r c_1(L,h^L)^n}+o(k^n).
 \end{equation}
	
  We have
		\begin{equation*}
		\sqrt{-1}R^{L_\chi}=\sqrt{-1}R^L+\sqrt{-1}\chi'(\varrho)\dbar\ddbar\varrho+\sqrt{-1}\chi''(\varrho)\dbar\varrho\wedge\ddbar\varrho\geq \sqrt{-1}R^L+\sqrt{-1}\chi'(\varrho)\dbar\ddbar\varrho.
		\end{equation*}
		Since $R^L$ has at least $n-s+1$ non-negative eigenvalues (thus at most $s-1$ negative eigenvalues) on $X\setminus M$, $\chi'(\varrho)\geq 0$ on $X$ and $\dbar\ddbar\varrho$ has at least $n-q+1$ positive eigenvalues (thus at most $q-1$ negative eigenvalues) on $X\setminus K$, the number of negative eigenvalues of $R^{L_\chi}$ is strictly less than $j$ on $X\setminus (M\cup K)$ for any $j\geq s+q-1$ (note $s+q-1>s-1$ and $>q-1$), and thus
		\begin{equation*}
		K'(j,h^L_\chi)\subset K\cup M\subset X_{u_0}.
		\end{equation*}

However, by $\chi=0$ on $(-\infty,u_0)$, we have $h^L_\chi=h^L$ on $X_{u_0}$ and $c_1(L,h^L_\chi)=c_1(L,h^L)$ on $X_{u_0}$. Thus $K'(j,h^L_\chi)=X_{u_0}(j,h_\chi^L)=X_{u_0}(j,h^L)=K'(j,h^L)\setminus (K'\setminus X_{u_0})(j,h^L)=K'(j,h^L)$ for $j\geq s+q-1$. It follows that
\begin{equation*}
 \int_{K'(\ge r,h^L_\chi)}(-1)^r c_1(L,h_\chi^L)^n
= \int_{K'(\ge r,h^L)}(-1)^r c_1(L,h^L)^n
\end{equation*}
for $q+s-1\leq r\leq n$.	 Finally, by (\ref{eq_h_xc}), it follows that for $q+s-1\leq r\leq n$, we have
	 \begin{equation}\label{eq_q_xc}
  \begin{split}
       \sum_{j=r}^n(-1)^{j-r}\dim_\mathbb{C}H_{(2)}^j(X_c,L^k)
       \leq &\frac{k^n}{n!}\int_{K'(\ge r,h^L)}{(-1)^r c_1(L,h^L)^n}+o(k^n)\\
       =&\frac{k^n}{n!}\int_{M(\ge r)}{(-1)^r c_1(L,h^L)^n}+o(k^n).
  \end{split}
 \end{equation}
    Here the last equality is from that $K'(j,h^L)=M(j,h^L)$ for any $j\ge r\ge q+s-1$.

		 By \cite[Theorem 3.5.6 (Hörmander),
	Theorem 3.5.7 (Andreotti-Grauert)]{MM07}, we have, for any $j\geq q$,
	$$H_{(2)}^{j}(X_c,L^k)\cong H^j(X_v,L^k)\cong H^j(X,L^k).$$ 
	Since $s\geq 1$, we can apply the above identification in \eqref{eq_q_xc} to complete our proof.
\end{proof}

\subsection{Complete manifolds}

A Hermitian manifold $(X,\omega)$ is called complete, if all geodesics are defined for all time on the underlying Riemannian manifold. 
In \cite{LSW}, the authors find optimal fundamental estimate on complete manifolds as follows.

\begin{lemma}[{\cite[Lemma 3.6]{LSW}}]\label{lem_complete}
	Let $(X,\omega)$ be a complete Hermitian manifold of dimension $n$. Let $(L,h^L)$ be a holomorphic Hermitian line bundle  on $X$ such that $\omega=c_1(L,h^L)$ on $X\setminus M$ for a compact subset $M$. Then there exist $C_0>0$ and $M\Subset M'$ such that for each $1\leq q\leq n$, we have for sufficiently large $k$,
	\begin{equation*}
	\left( 1-\frac{C_0}{k} \right)\|s\|^2\leq \frac{C_0}{k}\left( \|\ddbar^{K_X}_k s\|^2+\|\ddbar_k^{K_X*} s\|^2 \right)+\int_{M'}|s|^2dV_X
	\end{equation*}
	for $s\in\Dom(\ddbar^{K_X}_{k})\cap\Dom(\ddbar^{K_X*}_{k})\cap L^2_{n,q}(X,L^k)$.
\end{lemma}

\begin{proof}[Proof of Theorem \ref{thm_complete}]
    Let $q\geq 1$.
	From Lemma \ref{lem_complete},
	the optimal fundamental estimate holds in bidegree $(0,q)$ for forms with values in $L^k\otimes K_X$ for $k$ large.  Then Theorem \ref{prop_main} tells us that for any $1\le r\le n$, 
  \begin{equation*}
     \sum_{j=r}^n(-1)^{j-r}\dim_\mathbb{C}H_{(2)}^j(X,L^k\otimes K_X)\leq \frac{k^n}{n!}\int_{M'(\ge r)}{(-1)^r c_1(L,h^L)^n}+o(k^n).
 \end{equation*} 
  Noticed that $M(q)=M'(q)$ for any $q\ge 1$, we complete the proof.
\end{proof}

\bibliographystyle{amsalpha}

\end{document}